\documentclass[reqno,12pt]{amsproc}

\usepackage[top=3cm, bottom=3cm, left=2.54cm, right=2.54cm]{geometry}
\usepackage{amsmath,amsthm,amscd,amssymb,bbm,mathrsfs,latexsym,tikz,enumerate,verbatim,multirow,accents}
\usepackage[colorlinks,linkcolor=red,urlcolor=blue,citecolor=blue,pagebackref,hypertexnames=false]{hyperref}

\newcommand{\mR}{\mathbb{R}}

\newcommand{\mZ}{\mathbb{Z}}

\newlength{\dhatheight}

\def\Exp(#1){{\mathbb E}(#1)}
\def\implies{\Rightarrow}
\def\inpr#1,#2{\t \hbox{\langle #1 , #2 \rangle} \t}
\newcommand{\isoembed}{\stackrel{i}{\hookrightarrow}}
\def\ip<#1,#2>{\langle #1,#2 \rangle}

\def\norm#1{\left \Vert #1 \right \Vert}
\def\paren(#1){\left( #1 \right)}

\def\sparen(#1){\Bigl ( #1 \Bigr )}
\def\ssparen(#1){ (#1) }
\def\st{\thinspace : \thinspace}

\newcommand{\vece}{{\boldsymbol{e}}}

\newcommand{\vecx}{{\boldsymbol{x}}}
\newcommand{\vecz}{{\boldsymbol{z}}}
\newcommand{\vecy}{{\boldsymbol{y}}}
\newcommand{\vecu}{{\boldsymbol{u}}}

\newcommand{\bxi}{\boldsymbol{\xi}}

\newcommand{\zero}{\boldsymbol{0}}
\newcommand{\one}{\boldsymbol{1}}

\newcommand{\calA}{{\mathcal{A}}}

\newcommand{\calM}{{\mathcal{M}}}

\def\tv[#1,#2,#3]{{\begin{pmatrix} #1\\#2\\#3\end{pmatrix}}}
\def\fv[#1,#2,#3,#4]{{\begin{pmatrix} #1\\#2\\#3\\#4 \end{pmatrix}}}
\newcommand{\ellinf}[1]{\ell_{#1}^{\infty}}  

\numberwithin{equation}{section}
\theoremstyle{plain}
\newtheorem{theorem}{{\bf Theorem}}[section]
\newtheorem{lemma}[theorem]{{\bf Lemma}}
\newtheorem{corollary}[theorem]{Corollary}
\newtheorem{proposition}[theorem]{Proposition}

\theoremstyle{definition}
\newtheorem{definition}[theorem]{{\bf Definition}}

\newtheorem{question}[theorem]{Question}

\theoremstyle{remark}

\numberwithin{equation}{section}

\allowdisplaybreaks

\begin{document}
\title{Squaring the circle: embedding $S^1$ in $\ell_1$}

\author[Ian Doust]{Ian Doust}
\email{i.doust@unsw.edu.au}
\address{School of Mathematics and Statistics, University of New South Wales, Sydney, NSW 2052, Australia}

\author[Anthony Weston]{Anthony Weston}
\email{aweston2@andrew.cmu.edu}
\address{Mathematical and Physical Sciences, Carnegie Mellon University Qatar, Education City, Qatar}

\begin{abstract}
  Let $(S^1,\delta)$ be the unit circle endowed with the arc length metric. This paper concerns the
  question of which subsets of $(S^1,\delta)$ can be embedded isometrically into the sequence space
  $\ell_1$. It is well known that every finite subset of $S^1$ admits such an embedding, but the
  situation for infinite subsets, even including the whole circle, has perhaps been obscured by conflicting
  terminology in the literature. It is worth noting that the circle avoids the classical obstructions to
  isometric embeddability, since it is both of negative type and hypermetric.
  In this paper we show that if $X \subseteq S^{1}$ is a closed set and the Lebesgue measure of
  $X \cap X^{\ast}$ is positive, where $X^{\ast}$ is the antipodal set of $X$, then it is impossible to
  isometrically embed $(X, \delta)$ in $\ell_{1}$. As a result, no subset of $S^{1}$ with Lebesgue
  measure greater than $\pi$ can be isometrically embedded in $\ell_1$. Conversely, we show that any
  closed subset of $S^1$ whose intersection with any half-circle has measure zero can be isometrically
  embedded in $\ell_1$. These results have several consequences. They imply that
  the classical Banach space $L_1[0, 1]$, considered purely as a metric space, does not isometrically
  embed in $\ell_{1}$. Secondly, they yield a simple proof that a metric graph $(M, d)$
  embeds isometrically in $\ell_1$ if and only if it is a tree. In contrast to $S^{1}$,
  we show that some related metric spaces, such as the cylinder and the flat torus,
  contain finite subsets that cannot be isometrically embedded in $\ell_1$.
\end{abstract}

\maketitle

\section{Introduction}
The field of metric space embeddings has a long and extensive history. One important stream of ideas concerns determining necessary or sufficient conditions for a metric space $(X,d)$ to embed isometrically into one of the standard Banach spaces. Particular attention has been paid to embeddings into the sequence spaces $\ell_1$ and $\ell_2$, or function spaces such as $L_1$. 

One of the simplest non-Euclidean metric spaces is the unit circle $S^1$ endowed with the shortest path metric $\delta$. 
For various reasons the literature covering embedding properties of this simple space is not as clear as one might hope it to be. One reason for examining this space is that it avoids the standard obstructions for an isometric embedding in $\ell_1$. 
As will be noted in Section~\ref{Embed_S1}, it is easy to embed $(S^1,\delta)$ isometrically in $L_1(S^1)$, or to isometrically embed any \textit{finite} subset of $S^1$ in $\ell_1$. On the other hand, the whole space $(S^1,\delta)$ does not embed isometrically in $\ell_1$.
The majority of this paper concerns the question of just which subsets of $(S^1,\delta)$ can be embedded isometrically in $\ell_1$. We shall show that if  $X \subseteq S^1$ and the Lebesgue measure of the intersection of $X$ with its antipodal set $X^{\ast}$ is positive, then it is impossible to isometrically embed $(X,\delta)$ in $\ell_1$. This means that no subset with measure greater than $\pi$ can be isometrically embedded in $\ell_1$.  Conversely, we show that any closed subset of $S^1$ whose intersection with any half-circle has measure zero can be embedded isometrically in $\ell_1$.

These results have some interesting consequences. They provide a straightforward proof that not only is $L_1[0,1]$ not isometric to a subspace of $\ell_1$ as a Banach space, but this remains true even in the category of metric spaces.

In Section~\ref{Conseq} we prove some results about Riemannian manifolds of
negative type. In particular, we generalize a theorem of Hjorth, Kokkendorf and Markvorson on the $p$-negative type of product spaces where one of the factors is $S^1$. As a consequence, some simple metric spaces, such as the cylinder and the flat torus, contain finite subsets that cannot be isometrically embedded into $\ell_1$.

Finally, we shall note that the preceding results also yield a simple proof that a metric graph $(M, d)$ embeds isometrically into $\ell_1$ if and only if it is a tree.

In the next section we recall some basic facts and terminology concerning metric embeddings into classical Banach spaces, including $\ell_1$ and $L_1$.

\section{Background and notation}\label{sec:2}

Much of the work on $\ell_1$-embeddings has been motivated by problems arising in combinatorics and theoretical computer science and has therefore focused on finite metric spaces.
This has led to terminology that is sometimes awkward or easily misunderstood in the context of infinite spaces.
For instance, statements such as `the circle $S^1$ is $\ell_1$-embeddable' appear in several places in the literature. It is important then that we carefully define our terms. 

Given metric spaces $(X,d_X)$ and $(Y,d_Y)$ we shall write $(X,d_X) \isoembed (Y,d_Y)$ if there exists an isometric map $\gamma: (X,d_X) \to (Y,d_Y)$. We shall often omit the metrics when these are understood. As usual, for $p \ge 1$, we shall write $\ell_p^m$ for $\mathbb{R}^m$ endowed with the $\ell_p$ metric.

Consider the following conditions on a metric space $(X,d)$:
\begin{enumerate}
\item[(A)] $(X,d) \isoembed \ell_p^m$ for some $m$.
\item[(B)] For all finite subsets $Y \subseteq X$, $(Y,d) \isoembed \ell_p^m$ for some $m$.
\item[(C)] For all finite subsets $Y \subseteq X$, $(Y,d) \isoembed \ell_p$.
\item[(D)] $(X,d) \isoembed \ell_p$.
\end{enumerate}
For finite metric spaces, these four conditions are equivalent. The nontrivial implication is (D) $\implies$ (A), which is an immediate consequence of the following result (see \cite{Ba,Wi}).

\begin{theorem} Suppose $1 \le p < \infty$.
Let $X = \{\vecx_1,\dots,\vecx_n\}$ be a finite subset of $\ell_p$. Then there is a isometric map $\gamma: X \to \ell_p^m$ where $m = \frac{1}{2}n(n-1)$.
\end{theorem}

A consequence of this is that conditions (B) and (C) are in fact equivalent for spaces of any cardinality. The finite metric spaces that satisfy the above conditions are precisely those known as cut metrics; we refer the reader to \cite{DL} for a detailed discussion.

For infinite metric spaces, conditions (A), (B) and (D) are not equivalent. We shall largely follow \cite[Section 3]{DL} in our terminology, with the addition of a new term for condition~(B).

\begin{definition}\label{lp-embed-def}
Let $p \ge 1$. A metric space $(X,d)$ is said to be:

\begin{enumerate}
\item $\ell_p^m$-\textit{embeddable} if $(X,d) \isoembed \ell_p^m$.
\item \textit{$\ell_p$-embeddable} if it satisfies (A).
\item \textit{finitely $\ell_p$-embeddable} if it satisfies (B) (or equivalently (C)).
\item \textit{$\ell_p^\infty$-embeddable} if it satisfies (D).
\item \textit{$L_p$-embeddable} if $(X,d) \isoembed L_p(\Omega,\calA,\mu)$ for some measure space $(\Omega,\calA,\mu)$.
\end{enumerate}
\end{definition}

The reader should note that some of the literature uses $L^1$-embeddable for what we call  $\ell_1$-embeddable here; see for example \cite{AD,MM}.  

Our primary concern in this paper is in the case $p = 1$.
As we shall see, $(S^1,\delta)$ is finitely $\ell_1$-embeddable and $L_1$-embeddable, but neither $\ell_1$-embeddable nor $\ell_1^\infty$-embeddable. Clearly $\ell_1$ is $L_1$-embeddable. A little less obvious is the following.

\begin{theorem}\label{OsL1} \cite[Lemma 1.24]{Os} For any measure space $(\Omega,\calA,\mu)$, $L_1(\Omega,\calA,\mu)$ is finitely $\ell_1$-embeddable.
\end{theorem}

Many of the conditions which arise in this area are `finitely determined', which is to say that, as in conditions (B) and (C), they depend on properties of arbitrary finite subsets of a space. Examples of this are the conditions of being of $p$-negative type, or being hypermetric. 

\begin{definition}\label{type-defs2}
Let $(X,d)$ be a finite metric space with
$X = \{ x_{1}, x_{2}, \ldots x_{n} \}$, and let
$p \geq 0$.
\begin{enumerate}
\item We say that the space $(X,d)$ is of \textit{$p$-negative type} if
   \begin{equation}\label{Dxixi2} \sum_{i,j} d(x_i,x_j)^p \xi_i \xi_j \le 0
   \end{equation}
whenever $\bxi = (\xi_1,\dots,\xi_n) \in \mR^n$ satisfies $\sum_i \xi_i = 0$. If the inequality in (\ref{Dxixi2}) is strict unless $\bxi = \zero$, then $(X,d)$ is said to be of \textit{strict $p$-negative type}.

\item We say that the space $(X,d)$ is $p$-\textit{hypermetric} if (\ref{Dxixi2}) holds whenever $\bxi \in \mZ^n$ satisfies $\sum_i \xi_i  =1$.
\end{enumerate}
\end{definition}

For a general metric space $(X,d)$, we say that it is of $p$-\textit{negative type}, of strict $p$-\textit{negative type}, or $p$-\textit{hypermetric} if every finite subset of $(X,d)$ has the corresponding property. As is customary, a metric space of $1$-negative type is usually just said to be of \textit{negative type}. Similarly, a metric space that is $1$-hypermetric is usually just said to be \textit{hypermetric}.

For any given metric space $(X, d)$, the set of all $p$ for which $(X, d)$ has $p$-negative type is a (possibly degenerate) closed interval of the form $[0, \wp(X, d)]$, where $\wp(X, d) \geq 0$, or it is $[0, \infty)$, in which case $\wp(X, d)$ is taken to be $\infty$. Indeed all possible values $0 \le \wp(X,d) \le \infty$ are obtainable. The fact that $p$-negative type holds on closed intervals is due to Schoenberg \cite{Sc2}.
The first example of a metric space $(X, d)$ for which $\wp(X, d) = 0$ was given some thirty years later by Enflo \cite{En2}. The quantity $\wp(X,d)$ is known as the \textit{supremal $p$-negative type} of $(X, d)$. For finite metric spaces this quantity can be determined using S{\'a}nchez's formula.  (An alternative formula is given in \cite{Rob}.)

\begin{theorem}\label{Sanchez} \cite[Corollary 2.4]{San}
  Let $(X,d)$ be a finite metric space with $X = \{ x_{1}, x_{2}, \ldots x_{n} \}$. Then
  \[ \wp(X,d) = \min\{p \ge 0 \st \det(D_p) = 0 \text{ or }
             \ip<D_p^{-1} \one,\one> = 0 \}, \]
where $D_p = (d(x_i,x_j)^p)_{i,j = 1}^n$ is the $p$-distance matrix of the space and $\one = (1,1,\dots,1)^T \in \mR^n$.
\end{theorem}

The negative type and hypermetric properties provide the most familiar classical obstructions to the existence of $\ell_1$ embeddings. As is well-known, $\ell_1$ and $L_1[0,1]$ are of negative type and hypermetric, with $\wp(\ell_1) = \wp(L_1[0,1]) = 1$. Combining these facts with Theorem~\ref{OsL1} and standard properties of negative type and hypermetric spaces yields the following theorem.

\begin{theorem}\label{Properties-Scale} Let $(X,d)$ be a metric space. Then we have the following chain of implications:
\begin{align*}
\text{$(X,d)$ is $\ell_1$-embeddable}
& \implies \text{$(X,d)$ is $\ell_1^\infty$-embeddable} \\
& \implies \text{$(X,d)$ is $L_1$-embeddable} \\
& \implies \text{$(X,d)$ is finitely $\ell_1$-embeddable} \\
& \implies  \text{$(X,d)$ is hypermetric} \\
& \implies  \text{$(X,d)$ is of negative type}.
\end{align*}
\end{theorem}

It is well known that certain embedding properties are also finitely determined. The following result is stated for $\ell_p^m$ spaces in \cite{DL}, although the proof, which depends on a compactness argument due to Malitz and Malitz \cite{MM}, works for any real finite-dimensional normed space.

\begin{theorem} A metric space $(X,d)$ embeds isometrically in a real finite-dimensional normed space $(Z, \| \cdot \|)$ if and only if each finite subset of $(X, d)$ embeds isometrically in that space. 
\end{theorem}

In 1966, Bretagnolle, Dacunha-Castelle and Krivine \cite[page 252]{BDK} reduced the study of $L_{p}$-embeddable metric spaces to the finite case. We remark that the case $p = 2$ is due to Schoenberg \cite{Sc1}.

\begin{theorem}\label{BDK-Thm}
Suppose that $p \ge 1$ and that $(X, d)$ is a metric space. Then $(X, d)$ is $L_{p}$-embeddable if and only if every finite subset of $(X, d)$ is
$L_{p}$-embeddable.
\end{theorem}

The terminology here is delicate insofar as the underlying measure spaces may change. In particular (as we shall see below) even if there is some fixed measure space $(\Omega,\mathcal{A}, \mu)$ such that every finite subset of a metric space $(X,d)$ embeds isometrically in $L_p(\Omega,\calA, \mu)$,  it may be that $(X,d)$ itself does not embed isometrically in $L_p(\Omega,\calA, \mu)$.

Thus, while the conditions of a metric space being $\ell_{1}^{m}$-embeddable or $L_{1}$-embeddable (and of course of being finitely $\ell_1$-embeddable) are finitely determined, the condition of a metric space being $\ellinf{1}$-embeddable is not finitely determined.
One source of easy obstructions arise from cardinality.
The space $\ell_1[0,1]$, that is $L_1([0,1],\mu)$ with $\mu$ being counting measure, is clearly $\ell_1$-embeddable, but not $\ellinf{1}$-embeddable. 

One might avoid this particular obstruction by restricting one's attention to metric spaces $(X,d)$ which are separable. With this in mind we make the following definition.

\begin{definition}
Let $Z$ be a Banach space. We say that $Z$ has the \textit{separable finite embedding property} if, whenever every finite subset of a separable metric space $(X, d)$ embeds isometrically in $Z$, the entire space $(X, d)$ embeds isometrically in Z.
\end{definition}

It follows immediately that every finite-dimensional Banach space has this property. A more interesting example is provided by $L_1[0,1]$.

\begin{theorem}
$L_1[0,1]$ has the separable finite embedding property.
\end{theorem}

\begin{proof}
Suppose that $(X,d)$ is a separable metric space and every finite subset of $(X,d)$ embeds isometrically in $L_1[0,1]$. Then by Theorem~\ref{BDK-Thm} there is a measure space $(\Omega,\calA,\mu)$ such that $(X,d)$ is isometric to a separable subset $Y$ of $L_1(\Omega,\calA,\mu)$, and hence $Y$ lies is a separable subspace of that space. But every separable subspace of an $L_1$-space is linearly isometric to a subspace of $L_1[0,1]$. 
\end{proof}

In Section~\ref{Embed_S1} we shall see that one can find countable metric spaces which
are finitely $\ell_1$-embeddable, but not $\ellinf{1}$-embeddable, and so $\ell_1$ does not have the separable finite embedding property.

\section{Embedding subsets of the circle}\label{Embed_S1}

In terms of isometric embeddings, the spherical distance spaces $S^m$ form an interesting test case, so before progressing further we fix some notation. We shall let $S^m$ denote the unit sphere in $\mR^{m+1}$. Associated to $S^m$ is a geodesic, or great circle metric, which we denote $\delta$. 

The following result, Theorem~6.4.5 of \cite{DL}, tells us about the embedding properties of \textit{finite} subsets of these metric spaces (a restriction that is omitted from the statement of the theorem in \cite{DL}; this has propagated to some later works such as \cite[p.~140]{LP}).

\begin{theorem}\label{Sm_finitely_embeds}
For all $m \geq 1$, the spherical distance space $S^m$ is finitely $\ell_1$-embeddable. 
\end{theorem}

The example we shall be most interested in here is the unit circle $S^1$, in which case $\delta$ is the arc-length metric around the circle. In what follows we shall generally identify points on $S^1$ with their argument in $[0,2\pi)$.
The space $(S^1,\delta)$ avoids the standard obstructions for being isometrically embeddable in $\ell_1$. For example it is easy to embed the space into $L_1(S^1)$. One just maps $x \in S^1$ to $\frac{1}{2} \chi_{I_x}$, where $I_x$ is the set  $ \{y \in S^1 \st \delta(x,y) \le \pi/2\}$. This implies that $(S^1,\delta)$ is hypermetric, and hence also of negative type. 
More generally, $(S^1,\delta)$ embeds isometrically into any given nonzero
$L_{1}$-space that is not purely atomic, simply because such a space must contain an isometric copy of $L_1(S^1)$.
We note that it is relatively easy to give an explicit embedding of any finite subset of  $(S^1,\delta)$ into $\ell_1$ (without recourse to the slightly more involved theory of cut metrics). Indeed, if $X$ is an $n$-point subset of $S^1$, then (in the terminology of \cite{DW2026}) one can isometrically embed it as the vertices of a weighted Hamming cube $\prod_{k=1}^n \{0,w_k\} \subseteq \ell_1^n$. This embedding will appear as a special case of Theorem~\ref{halfcircle_plus} below.
And of course one can easily give a bi-Lipschitz embedding of $(S^1,\delta)$ into $\ell_1^2$. 

As we shall shortly see, however, $(S^{1}, \delta)$ is not $\ell_1^\infty$-embeddable. That raised the question of determining just which metric subspaces $(X, \delta)$ of $(S^{1}, \delta)$ do embed isometrically
in $\ell_{1}$, and that is focus of the rest of this section.

Beyond the fact that every finite subset of $S^1$ embeds isometrically in $\ell_{1}$, we record three trivial observations:
\begin{enumerate}
 \item A subset $X$ of $S^1$ embeds isometrically in $\ell_{1}$ if and only if its closure does. 
 \item Any subset of a half circle embeds isometrically in $\ell_1^1$, and hence in $\ell_1$.
 \item Up to a scaling factor, $(S^1,\delta)$ is isometric to any parameterized loop, so with suitable adjustments, all the embedding results below also apply to such loops.
\end{enumerate}
In what follows then we shall largely restrict our attention to closed subsets of $S^1$.

Let $m$ denote the one-dimensional Lebesgue measure on $S^{1}$. If $t \in [0,2\pi)$ we will denote its antipodal point by ${t^{\ast}}$.
Given a set $Y \subseteq S^{1}$,  we let $Y^{\ast} = \{{t^{\ast}} \st t \in Y\}$ be its antipodal set. 

\begin{theorem}\label{halfcircle_plus}
Suppose that $Y \subseteq S^1$ is a set of measure zero and that $X$ is a closed set with $Y \subseteq X \subseteq Y \cup [\pi,2\pi]$. Then $(X,\delta)$ can be embedded isometrically in $\ell_1$.
\end{theorem}

\begin{proof}
We begin with a few simplifications. It clearly suffices to consider the maximal case where $[\pi,2\pi] \subseteq X$, so we shall assume this. Note that this ensures that $0,\pi \in X$. The hypotheses then would still hold if $Y$ were replaced by $Y' = \{0,\pi\} \cup (Y \cap (0,\pi))$. Consequently we may further assume that $\{0,\pi\} \subseteq Y \subseteq [0,\pi]$. Since $X$ is closed, under this assumption we have $\overline{Y} \subseteq X \cap [0,\pi] \subseteq Y$ and hence $Y$ must be closed.
It follows then that $U = [0, \pi] \setminus Y$ can be expressed as a countable
  disjoint union of open intervals $U = \bigcup_{j \in J} U_{j}$, where
  $J = \{ 1, 2, \ldots, j_{0}\}$ for some $j_{0} \in \mathbb{Z}^{+}$ or $J = \mathbb{Z}^{+}$.

  Letting $U_{j} = (a_j,b_j)$ and $m_j  =b_j - a_j$, so that  $m(U) = \sum_j m_j = \pi$, we set
  $\vecu = \sum\limits_{j \in J} m_j \vece_j$. For $t \in [0,\pi]$, we set $I_t = \{j \st b_j \le t\}$, and for
  $t \in [\pi,2\pi)$, we set $\hat{I}_t = \{j \st b_j \le t^{\ast} \}$. These sets satisfy
  the following easily verified properties: $I_{0} = \emptyset = \hat{I}_{\pi}$, $\hat{I}_{t^{\ast}} = I_{t}$ if $t \in (0, \pi]$,
  and $I_{t^{\ast}} = \hat{I}_{t}$ if $t \in [\pi, 2\pi)$.

  Noting that $X$ can be expressed as an essentially disjoint union $Y \cup Y^{\ast} \cup \bigl( \cup_{j} \,U^{\ast}_j \bigr)$,
  wherein the only overlap of the underlying sets is the intersection $Y \cap Y^{\ast} = \{ 0, \pi \}$,
  we define a map $\gamma: X \to \ell_1$ as follows:
  \[
  \gamma(t)=
  \begin{cases}
    \displaystyle \sum_{j\in I_t} m_j \vece_j,
    & t\in Y,\\[1ex]

    \displaystyle \vecu-\sum_{j\in \hat{I}_t} m_j \vece_j,
    & t\in Y^{\ast},\\[1ex]

    \displaystyle \vecu-(t- a^{\ast}_\ell)\vece_\ell
    -\sum_{j\in \hat{I}_t} m_j \vece_j,
    & t\in U^{\ast}_\ell \text{ for some } \ell,
  \end{cases}
  \]
  where sums indexed over the empty set are taken to be the zero vector $\boldsymbol{0} \in \ell_{1}$.
  In particular, we have $\gamma(0) = \boldsymbol{0}$ and $\gamma(\pi) = \vecu$.

  The above mentioned properties of the sets $I_{t}$ and $\hat{I}_{t}$ ensure that
  $\gamma(t^{\ast}) = \vecu - \gamma(t)$ for all $t \in Y \cup Y^{\ast}$. As a result,
  we see that
  \begin{align*}
    \norm{  \gamma(t) - \gamma(t^{\ast}) }_{1}
    & =  \norm{ \gamma(t) - (\vecu - \gamma(t)) }_{1} \\
    & =  \Bigl\| \sum_{j \in I_t} m_j \vece_j - \sum_{j \in J \setminus I_t} m_j \vece_j \Bigr\|_{1} \\
    & =  \sum_{j \in I_t} m_j + \sum_{j \in J \setminus I_t} m_j \\
    & =  \sum_{j \in J} m_j \\
    & =  \pi = \delta(t,t^{\ast}),
  \end{align*}
  for all $t \in Y$. More generally then, we need to verify that $\delta(t, s) = \norm{  \gamma(t) - \gamma(s) }_{1}$ for all $s, t \in X$. This degenerates
  into the consideration of cases.

  Case 1: $0 \leq s < t \leq \pi$ where $s, t \in Y$

  Before continuing, it is helpful to define $I_{s, t} = I_{t} \setminus I_{s}$. It is then the case that
  $\gamma(t) - \gamma(s) = \sum_{j \in I_{s, t}} m_j \vece_j$, and so
  $\| \gamma(t) - \gamma(s) \|_{1} = \sum_{j \in I_{s, t}} m_j$. On the other hand
  $\delta(t, s) = t - s = m \left((s, t)\right)$. However, no $U_{j}$ can overlap $s$ or $t$, so
  \[
  (s, t) = \left( \cup_{j \in I_{s, t}} U_{j} \right) \cup \left( Y \cap (s, t) \right),
  \]
  and the latter set $Y \cap (s, t)$ has measure zero. Hence
  \[
  \delta(t, s) = m((s, t)) = \sum_{j \in I_{s, t}} m_j = \| \gamma(t) - \gamma(s) \|_{1},
  \]
  as asserted. This completes Case 1.

  Case 2: $\pi \leq s < t < 2 \pi$

  In this case there are four sub-cases, depending upon where $s$ and $t$ are located, relative to the sets
  $Y^{\ast}$ and $U^{\ast}$.

  (a) $s, t \in Y^{\ast}$
  
  (b) $s \in Y^{\ast}$, $t \in U^{\ast}$
  
  (c) $s \in U^{\ast}$, $t \in Y^{\ast}$
  
  (d) $s, t \in U^{\ast}$

  In order to verify Case 2 (a), suppose that $\pi \leq s < t < 2 \pi$ with $s, t \in Y^{\ast}$.
  Since $\gamma(t) = \vecu - \gamma(t^{\ast})$ and $\gamma(s) = \vecu - \gamma(s^{\ast})$, we may
  simply apply Case 1 to obtain the desired conclusion:
  \[ 
    \norm{  \gamma(t) - \gamma(s) }_{1}
      =   \| \gamma(s^{\ast}) - \gamma(t^{\ast}) \|_{1} 
      =   \delta(s^{\ast}, t^{\ast})  
      =   \delta(t, s).
  \] 

  In order to verify Case 2 (b), suppose that $\pi \leq s < t < 2 \pi$ with $s \in Y^{\ast}$ and $t \in U^{\ast}$.
  In this sub-case, $t \in U^{\ast}_\ell = (a^{\ast}_{\ell}, b^{\ast}_{\ell})$ for some $\ell \in J$. For notational
  simplicity, let $\alpha = a^{\ast}_{\ell}$ and $\beta = b^{\ast}_{\ell}$.
  As $t^{\ast} \in U_{\ell} = (\alpha^{\ast}, \beta^{\ast}) = (a_{\ell}, b_{\ell})$, it follows that
  $\hat{I}_{t} = \hat{I}_{\alpha}$. Keeping this in mind, we see that
  \begin{align*}
    \gamma(t) - \gamma(s)
    & =  \Bigl( \vecu - (t - \alpha)\vece_{\ell} - \sum_{j \in \hat{I}_{t}} m_{j} \vece_{j}\Bigr)
    - \left( \vecu - \gamma(s^{\ast} )\right) \\
    & =  - (t - \alpha)\vece_{\ell} - \sum_{j \in \hat{I}_{t}} m_{j} \vece_{j} + \sum_{j \in I_{s^{\ast}}} m_{j} \vece_{j} \\
    & =  - (t - \alpha)\vece_{\ell} - \sum_{j \in \hat{I}_{t}} m_{j} \vece_{j}	+ \sum_{j \in \hat{I}_{s}} m_{j} \vece_{j} \\
    & =  - (t - \alpha)\vece_{\ell} - \Bigl( \sum_{j \in \hat{I}_{\alpha}} m_{j} \vece_{j}
    - \sum_{j \in \hat{I}_{s}} m_{j} \vece_{j} \Bigr) \\
    & =  - (t - \alpha)\vece_{\ell} - \sum_{j \in I_{s^{\ast}, \alpha^{\ast}}} m_{j} \vece_{j},
  \end{align*}
  where the last line follows trivially because $\hat{I}_{\alpha} = I_{\alpha^{\ast}}$ and $\hat{I}_{s} = I_{s^{\ast}}$.
  Consequently,
  \[
  \norm{  \gamma(t) - \gamma(s) }_{1} 
    = (t - \alpha) 
        + \sum_{j \in I_{s^{\ast}, \alpha^{\ast}}} m_{j}.
  \]
  Furthermore, by Case 2 (a) and Case 1,
  \[ 
    \alpha - s
      =  \norm{  \gamma(\alpha) - \gamma(s) }_{1}  
      =  \norm{\gamma(s^{\ast}) - \gamma(\alpha^{\ast}) }_{1}  
      =  \sum_{j \in I_{s^{\ast}, \alpha^{\ast}}} m_{j}.
  \] 
  So, in fact, $\| \gamma(t) - \gamma(s) \|_{1} = (t - \alpha) + (\alpha - s) = t - s = \delta(t, s)$. This completes Case 2 (b).

  Cases 2 (c) and 2 (d) are similar and are omitted.

  Case 3: $0 \leq s < \pi < t < 2\pi$ where $s \in Y$

  In this case there are four sub-cases, depending upon $s^{\ast}$ and $t$.

  (a) $t \in Y^{\ast}$, $s^{\ast} \leq t$

  (b) $t \in Y^{\ast}$, $s^{\ast} > t$

  (c) $t \in U^{\ast}$, $s^{\ast} \leq t$

  (d) $t \in U^{\ast}$, $s^{\ast} > t$

  In order to verify Case 3 (a), suppose that $0 \leq s < \pi < t < 2\pi$, $s \in Y$, $t \in Y^{\ast}$ and $s^{\ast} \leq t$.
  In this sub-case, $\delta(t, s) = 2\pi - (t - s)$. As $s, t^{\ast} \in Y$, it is worth noting that
  $s = \| \gamma(s) - \gamma(\boldsymbol{0}) \|_{1} = \sum_{j \in I_{s}} m_{j} \vece_{j}$
  and (similarly) $t^{\ast} = \sum_{j \in I_{t^{\ast}}} m_{j} \vece_{j}$ by Case 1. As a result,
  \begin{eqnarray*}
    \gamma(t) - \gamma(s)
    & = & \vecu - \gamma(t^{\ast}) - \gamma(s) \\
    & = & \vecu - \sum_{j \in I_{t^{\ast}}} m_{j} \vece_{j} - \sum_{j \in I_{s}} m_{j} \vece_{j} \\
    & = & \sum_{j \notin I_{t^{\ast}}} m_{j} \vece_{j} - \sum_{j \in I_{s}} m_{j} \vece_{j}.
  \end{eqnarray*}
  However, $I_{s} \subseteq I_{t^{\ast}}$ because $s^{\ast} \leq t$. Hence, by the previous calculation,
  \begin{eqnarray*}
    \| \gamma(t) - \gamma(s) \|_{1}
    & = & \sum_{j \notin I_{t^{\ast}}} m_{j} + \sum_{j \in I_{s}} m_{j} \vece_{j} \\
    & = & \Big( \pi - \sum_{j \in I_{t^{\ast}}} m_{j} \Big) + s \\
    & = & (\pi - t^{\ast}) + s \\
    & = & (\pi - (t - \pi)) + s \\
    & = & 2\pi - (t - s) \\
    & = & \delta(t, s).
  \end{eqnarray*}
  This completes Case 3 (a).

  In order to verify Case 3 (b), suppose that $0 \leq s < \pi < t < 2\pi$, $s \in Y$, $t \in Y^{\ast}$ and $s^{\ast} > t$.
  In this sub-case we have $0 \leq t^{\ast} < \pi < s^{\ast} < 2\pi$ with $t^{\ast} \in Y$, $s^{\ast} \in Y^{\ast}$ and
  $t = (t^{\ast})^{\ast} < s^{\ast}$. So we can apply Case 3 (a) with $s$ replaced by $t^{\ast}$ and $t$ replaced by $s^{\ast}$. Hence,
  \begin{eqnarray*}
    \delta(t, s)
    & = & \delta(s^{\ast}, t^{\ast}) \\
    & = & \| \gamma(s^{\ast}) - \gamma(t^{\ast}) \|_{1} \\
    & = & \| \gamma(t) - \gamma(s) \|_{1},
  \end{eqnarray*}
  as one would hope. This completes Case 3 (b).

  Cases 3 (c) and 3 (d) are similar and are omitted.
\end{proof}

\begin{theorem}\label{antipodes}
  Let $X = Y \cup Y^{\ast} \subseteq S^{1}$ where $Y$ is a nonempty closed subset of $[0, \pi]$.
  Then $(X, \delta)$ can be isometrically embedded in $\ell_{1}$ if and only if $m(Y) = 0$.
\end{theorem}

\begin{proof} ($\Rightarrow$)
  The rotational invariance of $S^{1}$ allows us to assume that $0 \in Y$ and (hence) $\pi \in Y^{\ast}$.
  We may further assume that $0 \in Y^{\ast}$ (so that $Y \cap Y^{\ast} = \{ 0, \pi\}$).
  As $Y^{\ast}$ is a closed subset of $[\pi, 2\pi]$, $U = [\pi, 2\pi] \setminus Y^{\ast}$ can be expressed
  as a countable union of pairwise disjoint open intervals $U_{j} = (a_{j}, b_{j})$, $j \in J$, where
  $J = \{ 1, 2, \ldots, j_{0}\}$ for some $j_{0} \in \mathbb{Z}^{+}$ or $J = \mathbb{Z}^{+}$. By
  setting $m_{j} = b_{j} - a_{j}$ for each $j \in J$, it is then immediate that $m(U) = \sum\limits_{j \in J} m_{j}$.

  Suppose that $\gamma : X \rightarrow \ell_{1}$, $\gamma(t) = (x_{1}(t), x_{2}(t), \ldots)$, is an
  isometric embedding of $(X, \delta)$ in $\ell_{1}$. As $\gamma$ is an isometry, each component function
  $x_{i}(t)$ is non-expansive and hence continuous. Without loss of generality, we may assume that
  $\gamma(0) = (0, 0, 0, \ldots)$. Setting $\gamma(\pi) = \boldsymbol{u} = (u_{1}, u_{2}, u_{3}, \ldots) \in \ell_{1}$,
  we may further assume that $u_{i} \geq 0$ for all $i$ and that $u_{1} > 0$. So we see that
  $\pi = \| \gamma(\pi) - \gamma(0) \|_{1} = \sum\limits_{i} u_{i}$.

  We shall see presently that the component functions $x_{i}$ are non-negative, non-decreasing on $Y$, and
  non-increasing on $Y^{\ast}$. For instance, suppose that $x_{i}(t) < 0$ for some $i$ and some $t \in Y$. Then,
  \begin{eqnarray*}
    \pi & = & t + (\pi - t) \\
    & = & \| \gamma(t) - \gamma(0) \|_{1} + \| \gamma(\pi) - \gamma(t) \|_{1} \\
    & = & -x_{i}(t) + \sum\limits_{j \not= i} |x_{j}(t)| + (u_{i} - x_{i}(t)) + \sum\limits_{j \not= i} |u_{j} - x_{j}(t)| \\
    & \geq & -2x_{i}(t) + u_{i} + \sum\limits_{j \not= i} u_{j} \\
    & = & -2x_{i}(t) + \pi \\
    & > & \pi,
  \end{eqnarray*}
  which is a contradiction. Thus $x_{i}(t) \geq 0$ for all $i$ and all $t \in Y$.

  Now suppose that $s, t \in Y$ with $s < t$. As $t = s + (t - s)$ and each $x_{i}$ is non-negative on $Y$,
  it follows that
  \begin{eqnarray*}
    \sum\limits_{i} x_{i}(t)
    & = & \sum\limits_{i} x_{i}(s) + \sum\limits_{i} |x_{i}(t) - x_{i}(s)|,
  \end{eqnarray*}
  which is only possible if $x_{i}(s) \leq x_{i}(t)$ for all $i$. Hence each $x_{i}$ is non-decreasing on $Y$.

  Applying the same arguments as one goes from $2\pi = 0$ to $\pi$ in the opposite direction shows that each
  $x_{i}$ is non-negative and non-increasing on $Y^{\ast}$. By drawing these strands together, we see that
  \begin{enumerate}
  \item $0 \le x_i(t) \le u_i$,
  \item $|x_i(t) - x_i(s)| \le u_i$,
  \item $x_i$ is non-decreasing on $Y$, and
  \item $x_i$ is non-increasing on $Y^{\ast}$,
  \end{enumerate}
  for all $s,t \in X$ and all $i$.

  Now let $s \in X \cap [0,\pi)$ be given, and so $s^{\ast} = s + \pi \in X$. It is then the case that
  \begin{align*}
    \pi
    & =   \sum\limits_{i} u_{i} \\
    & =   \norm{  \gamma(s^{\ast}) - \gamma(s) }_{1} \\
    & =   \sum\limits_{i} |x_{i}(s^{\ast}) - x_{i}(s)| \\
    & \leq  \sum\limits_{i} u_{i} \\
    & =   \pi.
  \end{align*}
  Applying (2) with $t = s^{\ast}$ we deduce that $|x_{i}(s^{\ast}) - x_{i}(s)| = u_{i}$ for all $i$.
  Furthermore, as a consequence of (1), we see that either
  \begin{enumerate}
  \item[(5)] $x_{i}(s) = 0$ and $x_{i}(s^{\ast}) = u_{i}$, or
  \item[(6)] $x_{i}(s) = u_{i}$ and $x_{i}(s^{\ast}) = 0$,
  \end{enumerate}
  for each $s \in X$ and each $i$. Notably, these conditions are mutually exclusive whenever $u_{i} > 0$.
  In particular, the range of $x_{i}$ is $\{ 0, u_{i} \}$ for all $i$.

  Let $I = \{ i : u_{i} > 0 \}$ and fix an element $i \in I$. The map defined by
  $h_{i}(t) = x_{i}(t) - x_{i}(t + \pi)$, $t \in Y^{\ast}$,
  is continuous on $Y^{\ast}$. By (5) and (6), $|h_{i}(t)| = u_{i} > 0$ for all $t \in Y^{\ast}$.
  If $t_{1} < t_{2}$ in $Y^{\ast}$, then $t_{1} + \pi < t_{2} + \pi$ in $Y$, so that
  $x_{i}(t_{1}) \geq x_{i}(t_{2})$ and $x_{i}(t_{1} + \pi) \leq x_{i}(t_{2} + \pi)$ by (3) and (4). As a result,
  if $t_{1} < t_{2}$ in	$Y^{\ast}$, then
  \begin{eqnarray*}
    h_{i}(t_{2}) - h_{i}(t_{1})
    & = & (x_{i}(t_{2}) - x_{i}(t_{1})) - (x_{i}(t_{2} + \pi) - x_{i}(t_{1} + \pi)) \\
    & \leq & 0.
  \end{eqnarray*}
  Hence $h_{i}$ is non-increasing on $Y^{\ast}$. Moreover, the range of $h_{i}$ is $\{ -u_{i}, u_{i} \}$ because
  $0, \pi \in Y^{\ast}$.

  As $i \in I$, $-u_{i} \not= u_{i}$, and so by continuity and range considerations, $A_{i} = h_{i}^{-1}(u_{i})$
  and $B_{i} = h_{i}^{-1}(-u_{i})$ are disjoint non-empty closed subsets of $[\pi, 2\pi]$ whose union is all of $Y^{\ast}$.
  In addition, we have that $a < b$ for all $a \in A_{i}$ and $b \in B_{i}$, because $h_{i}$ is non-increasing on $Y^{\ast}$
  with range $\{ -u_{i}, u_{i} \}$. Thus, setting $\alpha_{i} = \max A_{i}$ and $\beta_{i} = \min B_{i}$, it follows that
  $\alpha_{i} < \beta_{i}$. By construction, $(\alpha_{i}, \beta_{i}) \subseteq U = [\pi, 2\pi] \setminus Y^{\ast}$
  and $(\alpha_{i}, \beta_{i})$ is not contained in any larger interval in $U$. This ensures that
  $(\alpha_{i}, \beta_{i}) = U_{j} = (a_{j}, b_{j})$ for some uniquely determined $j \in J$. This defines
  a map $\phi : I \rightarrow J$. Before continuing, it is worth noting that in the current context,
  $h_{i}(a_{j}) = u_{i} = x_{i}(a_{j}) - x_{i}(a_{j} + \pi)$ and $h_{i}(b_{j}) = - u_{i} = x_{i}(b_{j}) - x_{i}(b_{j} + \pi)$.
  On account of (5) and (6), the noted equations force $x_{i}(a_{j}) = u_{i}$ and $x_{i}(b_{j}) = 0$.
  In summary, if $\phi(i) = j$, then it must be the case that $x_{i}(a_{j}) - x_{i}(b_{j}) = u_{i}$. This will be used below.

  It turns out that the map $\phi$ is onto. Indeed, let $j \in J$ be given. According to the various definitions, we
  must show that there exists an index $i_{j} \in I$ such that $\max h_{i_{j}}^{-1}(u_{i_{j}}) = a_{j}$ and
  $\min h_{i_{j}}^{-1}(-u_{i_{j}}) = b_{j}$. Now $a_{j}$ and $b_{j}$ are distinct elements of $Y^{\ast}$, so
  $\gamma(a_{j}) \not= \gamma(b_{j})$. Hence there must exist an index $i_{j}$ such that
  $x_{i_{j}}(a_{j}) \not= x_{i_{j}}(b_{j})$. Moreover, as $x_{i_{j}}$ is non-increasing on $Y^{\ast}$ with range
  $\{ 0, u_{i_{j}} \}$, we must have $x_{i_{j}}(a_{j}) = u_{i_{j}} > 0$ and $x_{i_{j}}(b_{j}) = 0$. Therefore $i_{j} \in I$.
  As $|h_{i_{j}}(t)| = u_{i_{j}} > 0$ for all $t \in Y^{\ast}$, we must have $h_{i_{j}}(a_{j}) = u_{i_{j}}$
  and $h_{i_{j}}(b_{j}) = - u_{i_{j}}$, because $h_{i_{j}}$ is non-increasing on $Y^{\ast}$ with range
  $\{ -u_{i_{j}}, u_{i_{j}}\}$. Since no elements of $Y^{\ast}$ lie between $a_{j}$ and $b_{j}$, we deduce that
  $\max h_{i_{j}}^{-1}(u_{i_{j}}) = a_{j}$ and $\min h_{i_{j}}^{-1}(-u_{i_{j}}) = b_{j}$. In other words,
  $\phi(i_{j}) = j$.

  As $\phi$ is onto,
  \begin{align*}
    \pi
    & =  \sum\limits_{i \in I} u_{i} \\
    & =  \sum\limits_{j \in J} \sum\limits_{i \in \phi^{-1}(j)} u_{i} \\
    & =  \sum\limits_{j \in J}	\sum\limits_{i \in \phi^{-1}(j)} x_{i}(a_{j}) - x_{i}(b_{j}) \\
    & \leq  \sum\limits_{j \in J} 
      \norm{ \gamma(b_{j}) - \gamma(a_{j}) }_{1} \\
    & =  \sum\limits_{j \in J} b_{j} - a_{j} \\
    & =  m(U) \\
    & \leq  \pi.
  \end{align*}
  We must therefore have $m(U) = \pi$ and hence $m(Y^{\ast}) = m(Y) = 0$.

  ($\Leftarrow$) This follows immediately from Theorem~\ref{halfcircle_plus}. 
\end{proof}

\begin{corollary}
  No circular arc in $S^{1}$ of length greater than $\pi$ can be isometrically embedded in $\ell_{1}$.
  More generally, $\pi$ is the maximum Lebesgue measure for which a closed subspace of $S^{1}$ can be
  isometrically embedded in $\ell_{1}$.
\end{corollary}

\begin{proof}
  Let $Z$ be a closed and Lebesgue measurable subset of $S^{1}$ such that $m(Z) > \pi$.
  Elementary measure theory ensures that $m(Z \cap Z^{\ast}) > 0$.
  It then follows from Theorem \ref{antipodes} that $(Z \cap Z^{\ast}, \delta)$, and hence $(Z, \delta)$,
  cannot be isometrically embedded in $\ell_{1}$.

That the maximum can be obtained follows from Theorem~\ref{halfcircle_plus}.
\end{proof}

\begin{corollary}\label{notS1}
$(S^1,\delta)$ is not $\ell_1^\infty$-embeddable.
\end{corollary}

A consequence of the above theorems is that while any closed subset of $S^1$ of Lebesgue measure zero (and in particular any countable closed subset of $S^1$) embeds isometrically in $\ell_1$, for any $\epsilon > 0$ there exists a closed set $X \subset S^1$ of measure $\epsilon$ which cannot be isometrically embedded. 

\section{Consequences}\label{Conseq}

\subsection{Embedding $L_1$ in $\ell_1$}
The circle $(S^1,\delta)$ then provides an example of a separable metric space for which every finite subset embeds isometrically in $\ell_1$, but for which the whole space does not. This shows that $\ell_1$ does not have the `separable finite subset embedding property' which is enjoyed by $L_1[0,1]$
and the finite dimensional spaces $\ell_1^n$. 
Of course it also indirectly gives a proof of the well-known fact that $L_1[0, 1]$ does not embed linearly isometrically in $\ell_1$. In fact we can say a little more.

\begin{theorem}
The classical Banach space $L_1[0, 1]$, considered purely as a metric space, does not isometrically embed into the sequence space $\ell_{1}$.
\end{theorem}

\subsection{Riemannian manifolds of negative type}

Theorem~5.4 of \cite{HKM} claims that every compact Riemannian manifold of negative type is simply connected. The space $(S^1,\delta)$, however, gives a counterexample to that statement, at least for one dimensional spaces. 
The results in \cite{HKM} do however show that simple connectedness does follow if one requires just a little more of the manifold.

\begin{theorem}\label{strict-connected}
A compact Riemannian manifold of strict negative type must be simply connected.
\end{theorem}

\begin{proof}
Let $(\calM,d)$ be a compact Riemannian manifold. Then there are geodesics between any two points of the space, and so the space is a `compact length space'. It follows immediately now from 
\cite[Corollary 5.2]{HKM} that if $(\calM,d)$ is of strict negative type then it is simply connected.
\end{proof}

The compactness assumption here is needed: the punctured plane is of strict negative type but not simply connected.

Theorem \ref{strict-connected} is closely related to a  result of Lafont and Prassidis \cite{LP}
(see Theorem~\ref{LP-Theorem} below) that links the simple connectedness of Riemannian manifolds to their roundness.
The concepts of roundness and generalized roundness of a metric space were introduced by Enflo~\cite{En1, En2, En3} in the late 1960s. 

\begin{definition}\label{round-defs} Let $(X,d)$ be a metric space and suppose that $p \ge 0$. 
\begin{enumerate}
  \item $(X,d)$ has \textit{roundness} $p$ if
  \[ d(a_1,a_2)^p + d(b_1,b_2)^ p 
       \le \sum_{1 \le i,j \le 2} d(a_i,b_j)^p\]
for every choice of points $a_1,a_2,b_1,b_2 \in X$.

  \item $(X,d)$ has \textit{generalized roundness} $p$ if
     \[ \sum_{1 \le i < j \le n} d(a_i,a_j)^p
        + \sum_{1 \le i < j \le n} d(b_i,b_j)^p
        \le \sum_{1 \le i,j \le n} d(a_i,b_j)^p \]
        for every integer $n \geq 2$ and every choice of points
        $a_1,\dots,a_n,b_1,\dots,b_n \in X$.
\end{enumerate}
\end{definition}

We note that generalized roundness $p$ implies roundness $p$ for every metric space $(X,d)$. For further discussion of these concepts, see Enflo~\cite{En1, En2, En3}. We also recall the theorem of Lennard, Tonge and Weston~\cite{LTW}, which asserts that a metric space $(X,d)$ has generalized roundness $p$ if and only if it is of $p$-negative type. Consequently, the supremal $p$-negative type $\wp(X,d)$ coincides with the supremum of the set of all $p \geq 0$ for which $(X,d)$ has generalized roundness $p$, the latter being the quantity considered by Enflo in~\cite{En2}.

\begin{theorem}\label{LP-Theorem} \cite{LP} A non-simply-connected, compact, Riemannian manifold has (maximal) roundness $1$. Equivalently,
a compact Riemannian
manifold with (maximal) roundness greater than $1$ must be simply connected.
\end{theorem}

It follows immediately from this that if a compact Riemannian manifold is of $p$-negative type for any $p > 1$ then it is simply connected. Note that an infinite metric space can be of strict $1$-negative type without being of $p$-negative type for any $p > 1$ (see \cite[Theorem~5.7]{DW2008}), so Theorem~\ref{strict-connected} strengthens this slightly.

The following result, Corollary 5.6 of \cite{HKM}, shows that other standard manifolds have rather different embedding properties to spheres. 

\begin{theorem}\label{HKM-products} Let $m \geq 1$. A Riemannian product manifold, where $S^m$ is one of the factors, is not of negative type.
\end{theorem}

In fact, the Riemannian structure is not needed at all here. Given a metric space $(X_0,d_0)$ we can form the $\ell_2$-product metric space $X = X_0 \times S^1$ with metric $d((x,t),(x',t')) = \sqrt{d_0(x,x')^2+\delta(t,t')^2}$.

\begin{theorem}
Suppose that $(X_0,d_0)$ is a metric space with at least two points. Then the $\ell_2$-product metric space $X = X_0 \times S^1$ has supremal $p$-negative type strictly less than $1$.
\end{theorem}

\begin{proof} 
Choose $x_0 \ne x_1 \in X_0$, and let $\alpha = \frac{2}{\pi} d_0(x_0,x_1) > 0$.
Let
  \[ Y = \{\vecy_i\}_{i=1}^5
     = \{ (x_0,0),(x_0,\textstyle{\frac{\pi}{2}}), (x_0,\pi),
         (x_0,\textstyle{\frac{3\pi}{2}}),(x_1,0) \}. \]
For $p \ge 1$, let $D_{Y,p}$ denote the $p$-distance matrix $D_{Y,p} = [d(\vecy_i,\vecy_j)^p]_{i,j=1}^5$. As for any 5 point metric space, $\det(D_{X,p}) \to 4$ as $p \to 0^+$. On the other hand one can verify that, with $A = \sqrt{\alpha^2+1}$ and $B = \sqrt{\alpha^2+4}$,
  \[ \left(\frac{2}{\pi}\right)^6 \det(D_{Y,1})
   = 16AB + 16A\alpha - 8B\alpha - 24\alpha^2 - 32
   =-(4A - 2B - 2\alpha)^2 \le 0.\]
This can only be zero if $2A = B+\alpha$. Squaring both sides and writing everything in terms of $\alpha$, this reduces to requiring that $\alpha = \sqrt{\alpha^2+4}$, which is clearly impossible.

Thus, $\det(D_{Y,1}) < 0$ for any $\alpha > 0$. As $p \mapsto \det(D_{Y,p})$ is continuous, there must exist $p_0 < 1$ such that $\det(D_{Y,p_0}) = 0$. It follows from S{\'a}nchez's formula (Theorem~\ref{Sanchez}) that the supremal $p$-negative type of $Y$ is at most $p_0$, and hence the same is true for $(X,d)$. 
\end{proof}

\begin{corollary}
If $(X_0,d_0)$ is a metric space with at least two points then the $\ell_2$-product metric space $X = X_0 \times S^1$ is not $L_1$-embeddable nor $\ell_1$-embeddable.
\end{corollary}

The above results apply to Riemannian manifolds such as the cylinder $\mR \times S^1$ and the flat torus $S^1 \times S^1$ (for which the Riemannian product metric agrees with the $\ell_2$-product used above). It would be of interest to know just what the supremal $p$-negative types of these two spaces are.
Using S{\'a}nchez's formula, one finds numerically that the supremal p-negative types of both spaces are less than 0.9533.

It is worth noting that it is straightforward to check that under the $\ell_1$-product metric, if $X$ and $Y$ are $\ell_1$ embeddable, then so is $X \times Y$. 
\subsection{Supremal $p$-negative type of subsets of $S^1$}

If $X \subseteq S^1$ then we must have $\wp(X,\delta)  \in [1,\infty]$. One might ask whether all such values can be obtained.
The first thing to note is that as soon as $X$ contains 3 points in any closed half circle, and in particular if $X$ contains 4 or more points, then $\wp(X,\delta) \le 2$. This is because one then has a subset $X_0 = \{x_1,x_2,x_3\} \subseteq X$ with $\delta(x_1,x_3) = \delta(x_1,x_2) + \delta(x_2,x_3)$, and this condition ensures that $\wp(X_0,\delta) = 2$. 

Among connected subsets of $S^1$, only the values $1$ and $2$ are possible.
Notably, this is another indication of the sensitivity of the supremal $p$-negative type to small perturbations of the metric. For example, as is shown in \cite{LTW}, $\wp(\ell_p^3) = p$ for $0 < p \le 2$, but $\wp(\ell_p^3) = 0$ for $p > 2$.

\begin{proposition}\label{mgr_connected} Let $m$ denote the one-dimensional Lebesgue measure on $S^{1}$. Suppose that $X$ is a connected subset of $S^1$. Then
  \[ \wp(X,\delta) = \begin{cases}
      2,  & \text{if $m(X) \le \pi$,} \\
      1,  & \text{if $m(X) > \pi$.}
      \end{cases} \]
\end{proposition}

\begin{proof} If $0 < t \le \pi$, then $(X_t,\delta)$ is isometric to the interval $[0,t]$ and hence $\wp(X_t,\delta) = 2$.

If $m(X) > \pi$ then $X$ must contain two pairs of antipodal points and this ensures that $\wp(X_t,\delta) \le 1$; see, for example, \cite[Section 2]{HKM}. Alternatively one can directly check that if $D$ is the corresponding distance matrix, then $\det(D) = 0$, and so this upper bound follows from
S{\'a}nchez's formula.
\end{proof}

If $X$ is a $3$ point set, then necessarily $(X,\delta)$ embeds isometrically in $\ell_2^2$ and so $\wp(X,\delta) \ge 2$. 
If $X_t = \Bigl\{0,t,2\pi-t \Bigr\}$ where $\frac{\pi}{2} \le t \le \frac{2\pi}{3}$, then the distance matrix for this space is 
  \[ D = \begin{pmatrix}
          0   & t   & t \\
          t   & 0   & 2\pi-2t \\ 
          t   & 2\pi-2t & 0
      \end{pmatrix}.  
            \]
The determinant of the $p$-distance matrix is 
  \[ \det(D_p) = 2\, t^{2p}\, (2 \pi - 2 t)^p,  \]
which is never zero. Inverting $D_p$ and adding the entries gives
  \[ \ip<D_p^{-1} \one,\one> = 0
  \iff (2 \pi - 2 t)^p - 4 t^p = 0. \]
For $\frac{\pi}{2} \le t < \frac{2\pi}{3}$ this has a first positive root at
  \[  \frac{2 \log 2}{\log(2 \pi - 2t) - \log t}, \]
which, by S{\'a}nchez's formula, is thus the value of $\wp(X_t,\delta)$. As $t$ increases from $\frac{\pi}{2}$ to $\frac{2\pi}{3}$, $\wp(X_t,\delta)$ increases continuously from $1$ to $\infty = \wp(X_{2\pi/3},\delta)$.

The question of which values in the interval $[1,2]$ are obtainable as $\wp(X,\delta)$ for some $X \subseteq S^1$ appears to be more difficult. If one chooses $n$ points equally spaced around the unit circle, one is essentially considering the cyclic graphs $C_n$. It is known \cite[Example~2.8]{San} that $\wp(C_5) = \log_2\Bigl(\frac{3+\sqrt{5}}{2}\Bigr) \approx 1.388$. More generally, by considering the roundness of these graphs, Horak et al. \cite[Corollary 3.2]{Hor} showed that $\wp(C_n) \to 1$ as $n \to \infty$. Scaling all the edges to have length $1$, and writing $D(n)$ for the distance matrix of $C_n$, one finds that for odd $n \ge 3$
  \[ \det(D(n)) = \frac{n^2-1}{4}, \qquad 
     \ip<D(n)^{-1} \one,\one> = \frac{4n}{n^2-1}. \]
Since neither of these can be zero and we know that $\wp(C_n) \ge 1$, this implies that $\wp(C_n) > 1$. It follows then that there are at least countably many values in $[1,2]$ which occur as $\wp(X,\delta)$ for some $X \subseteq S^1$. Numerical evidence suggests that in fact all values in $[1,2]$ are obtainable, and indeed obtainable from 4 point subsets of the circle, but we do not have a proof of this.
\begin{question}
If $1 < p < 2$, does there exist $X \subseteq S^1$ such that $\wp(X,\delta) = p$?
\end{question}

\subsection{Metric graphs}

Let $G = (V,E)$ be a connected edge-weighted graph. To avoid complications, we shall assume that $G$ has only finitely many vertices, and that there are no loops or multiple edges in the graph.

There are two natural ways to construct a metric space from the graph $G$.
The first is to consider the finite metric space $X_G$ whose points are the vertices of $G$, and for which the metric is the smallest
sum of edge weights for any path from one vertex to another.
A second possibility is to construct a compact connected metric space $M_G$ by considering each edge to be a part of the space. That is, for each $e = \{v,w\} \in E$ with weight (or length) $L_e$, we add a path of length $L_e$ to $X_G$ joining $v$ to $w$. Again we can use shortest paths to provide a metric $d_G$ on $M_G$. (Constructing $M_G$ more formally is of course possible, but that rarely provides better understanding; see for example \cite[Section~3.2.2]{BBI}.) If $G$ is a tree, then spaces of the second type are examples of separable $\mR$-trees.
It is well known that every separable $\mathbb{R}$-tree admits an isometric
embedding into $\ell_1$; see, for example, Aliaga \textit{et al}.\ \cite{Ali}.

The next lemma is an immediate consequence of \cite[Lemma 5.1]{HKM}.

\begin{lemma}
If $G$ contains a cycle then there exists $S \subseteq M_G$ such that $(S,d_G)$ is isometric to $(S^1,c \delta)$ for some $c > 0$.
\end{lemma}

\begin{theorem}
Let $G$ be as above. Then $(M_G,d_G)$ embeds isometrically in $\ell_1$ if and only if $G$ is a tree.
\end{theorem}

\begin{proof}
From the above comment, we just need to show that if $G$ is not a tree, then it does not embed isometrically in $\ell_1$.
So suppose $G$ is not a tree, and thus contains a cycle.
By the lemma there exists a set $S \subseteq M_G$ which is isometric to a circle with arc-length metric. But by Corollary~\ref{notS1}, $(S,d_G)$ does not embed isometrically in $\ell_1$, and so certainly neither does the whole space $(M_G,d_G)$. 
\end{proof}

\section{Higher dimensions}

The question of identifying which subsets of $(S^m,\delta)$ embed isometrically in $\ell_1$ when $m > 1$ appears to be more complicated, even in the case $m=2$. Accordingly, we shall focus primarily on that case.
Recall from Theorem~\ref{Sm_finitely_embeds} that $(S^m,\delta)$ is finitely $\ell_1$-embeddable for all $m \ge 1$, so the substantial question concerns infinite subsets of $S^m$.

Here we limit ourselves to a few observations and questions. Note that on the scale of properties included in Theorem~\ref{Properties-Scale}, $(S^2,\delta)$ is $L_1$-embeddable, but not $\ell_1^\infty$-embeddable. Indeed, using characteristic functions of hemispheres, one may readily embed $(S^2,\delta)$ into $L_1(S^2,\mu)$, where $\mu$ is the natural surface measure on the sphere. The results from Section~\ref{Embed_S1} obviously imply that if $X \subseteq S^m$ contains more than half of any great circle, then $X$ does not embed isometrically in $\ell_1$. 

Some obvious questions one may ask include the following.
\begin{enumerate}
  \item Which infinite subsets of $S^2$ that are not contained in a great circle embed isometrically in $\ell_1$?
  \item Does every countable closed subset of $S^2$ embed isometrically in $\ell_1$?
  \item Does any infinite connected subset of $S^2$ which is not contained in a great circle embed isometrically in $\ell_1$?
  \item Does any subset of $S^2$ of positive surface measure embed isometrically in $\ell_1$?
\end{enumerate}

In what follows we shall parameterize $S^2$ by spherical polar coordinates $(\theta,\phi)$ with $0 \le \theta < 2\pi$ and $0 \le \phi \le \pi$. 

If $X \subseteq S^2$ lies inside some great circle, then the results of Section~\ref{Embed_S1} apply. There are certainly infinite subsets of $S^2$ which are not subsets of a great circle which embed isometrically in $\ell_1$. 
For example, if $X$ consists of half the equatorial circle, together with the north and south poles, then one can isometrically embed $X$ in $\ell_1^2$ as the line from $(0,\frac{\pi}{2})$ to $(\frac{\pi}{2},0)$, together with the `poles' $(0,0)$ and $(\frac{\pi}{2},\frac{\pi}{2})$.

A more interesting example is the set
  \[ X = \left\{\vecx_k = \Bigl(\frac{\pi}{2^k},\frac{\pi}{2}\Bigr),\ \vecx^{\ast}_k = \Bigl(\pi+\frac{\pi}{2^k},\frac{\pi}{2} \Bigr) \st k = 0,1,2,\dots \right\} \cup \{\vecz = (0,0)\} \subseteq S^2, \]
which consists of countably many antipodal pairs on the equatorial circle, plus the north pole, which is equidistant from all of the other points. Following the algorithm given in the proof of Theorem~\ref{halfcircle_plus}, we can embed $\{\vecx_k,\vecx^{\ast}_k\}_{k=0}^\infty$ in $\ell_{1}$ by setting
   \begin{align*}
    \gamma(\vecx_k) &= \begin{cases}
     \big(\frac{\pi}{2}, \frac{\pi}{4}, \frac{\pi}{8},\dots\bigr), & k =0, \\
       \big(0,\dots,0,\frac{\pi}{2^{k+1}}, \frac{\pi}{2^{k+2}},\dots\bigr), & k \ge 1,
       \end{cases} \\
    \gamma(\vecx^{\ast}_k) &= \gamma(\vecx_0) - \gamma(\vecx_k), \qquad  k \ge 0.
   \end{align*}
One can then extend this to an isometric map from all of $(X,\delta)$ into $\ell_1$ by setting
  \[ \gamma(\vecz) = \Big(\frac{\pi}{4}, \frac{\pi}{8}, \frac{\pi}{16},\dots\Bigr).\]
This particular construction however would not allow one to embed $X \cup \{(0,\pi)\}$ isometrically into $\ell_{1}$, where we are adding in the south pole.

An intriguing class of metric spaces in this setting is provided by circles of constant latitude. For $0 < s < \frac{\pi}{2}$, let
    $S_s = \{(\theta,\phi) \in S^2 \st \phi = s\}$,
equipped with the metric inherited from $S^2$. It seems likely that as $s$ goes from $0$ to $\frac{\pi}{2}$, $\wp(S_s,\delta)$ decreases from $2$ to $1$, although we have no proof of this. Do any nontrivial subsets of $S_s$ embed isometricaly into $\ell_1$?

\bibliographystyle{amsalpha}

\end{document}